\documentclass[11pt]{article}

  \usepackage{a4wide}
  \usepackage[utf8]{inputenc}
  \usepackage{amsmath}
  \usepackage{amssymb}
  \usepackage{amsthm}
  \usepackage{mathrsfs}
  \usepackage{bbm}
  \usepackage{pifont}
  \usepackage{color}
  \usepackage{graphicx}
  \usepackage[normalem]{ulem}
	
	\def\AA{{\ifmmode{\mathbbm{A}}\else{$\mathbbm{A}$}\fi}}
	\def\BB{{\ifmmode{\mathbbm{B}}\else{$\mathbbm{B}$}\fi}}
	\def\CC{{\ifmmode{\mathbbm{C}}\else{$\mathbbm{C}$}\fi}}
	\def\EE{{\ifmmode{\mathbbm{E}}\else{$\mathbbm{E}$}\fi}}
	\def\FF{{\ifmmode{\mathbbm{F}}\else{$\mathbbm{F}$}\fi}}
	\def\HH{{\ifmmode{\mathbbm{H}}\else{$\mathbbm{H}$}\fi}}
	\def\KK{{\ifmmode{\mathbbm{K}}\else{$\mathbbm{K}$}\fi}}
	\def\NN{{\ifmmode{\mathbbm{N}}\else{$\mathbbm{N}$}\fi}}
	\def\PP{{\ifmmode{\mathbbm{P}}\else{$\mathbbm{P}$}\fi}}
	\def\QQ{{\ifmmode{\mathbbm{Q}}\else{$\mathbbm{Q}$}\fi}}
	\def\RR{{\ifmmode{\mathbbm{R}}\else{$\mathbbm{R}$}\fi}}
	\def\TT{{\ifmmode{\mathbbm{T}}\else{$\mathbbm{T}$}\fi}}
	\def\UU{{\ifmmode{\mathbbm{U}}\else{$\mathbbm{U}$}\fi}}
	\def\ZZ{{\ifmmode{\mathbbm{Z}}\else{$\mathbbm{Z}$}\fi}}
	
	\def\A{{\ifmmode{\mathscr{A}}\else{$\mathscr{A}$}\fi}}
	\def\B{{\ifmmode{\mathscr{B}}\else{$\mathscr{B}$}\fi}}
	\def\C{{\ifmmode{\mathscr{C}}\else{$\mathscr{C}$}\fi}}
	\def\D{{\ifmmode{\mathscr{D}}\else{$\mathscr{D}$}\fi}}
	\def\E{{\ifmmode{\mathscr{E}}\else{$\mathscr{E}$}\fi}}
	\def\F{{\ifmmode{\mathscr{F}}\else{$\mathscr{F}$}\fi}}
	\def\G{{\ifmmode{\mathscr{G}}\else{$\mathscr{G}$}\fi}}
	\def\H{{\ifmmode{\mathscr{H}}\else{$\mathscr{H}$}\fi}}
	\def\I{{\ifmmode{\mathscr{I}}\else{$\mathscr{I}$}\fi}}
	\def\J{{\ifmmode{\mathscr{J}}\else{$\mathscr{J}$}\fi}}
	\def\K{{\ifmmode{\mathscr{K}}\else{$\mathscr{K}$}\fi}}
	\def\L{{\ifmmode{\mathscr{L}}\else{$\mathscr{L}$}\fi}}
	\def\M{{\ifmmode{\mathscr{M}}\else{$\mathscr{M}$}\fi}}
	\def\N{{\ifmmode{\mathscr{N}}\else{$\mathscr{N}$}\fi}}
	\def\O{{\ifmmode{\mathscr{O}}\else{$\mathscr{O}$}\fi}}
	\def\P{{\ifmmode{\mathscr{P}}\else{$\mathscr{P}$}\fi}}
	\def\Q{{\ifmmode{\mathscr{Q}}\else{$\mathscr{Q}$}\fi}}
	\def\R{{\ifmmode{\mathscr{R}}\else{$\mathscr{R}$}\fi}}
	\def\S{{\ifmmode{\mathscr{S}}\else{$\mathscr{S}$}\fi}}
	\def\T{{\ifmmode{\mathscr{T}}\else{$\mathscr{T}$}\fi}}
	\def\U{{\ifmmode{\mathscr{U}}\else{$\mathscr{U}$}\fi}}
	\def\V{{\ifmmode{\mathscr{V}}\else{$\mathscr{V}$}\fi}}
	\def\W{{\ifmmode{\mathscr{W}}\else{$\mathscr{W}$}\fi}}
	\def\X{{\ifmmode{\mathscr{X}}\else{$\mathscr{X}$}\fi}}
	\def\Y{{\ifmmode{\mathscr{Y}}\else{$\mathscr{Y}$}\fi}}
	\def\Z{{\ifmmode{\mathscr{Z}}\else{$\mathscr{Z}$}\fi}}

	\newtheoremstyle{slanted}
	{}
	{}
	{\slshape}
	{}
	{\bfseries}
	{.}
	{ }
	{}
	
	\theoremstyle{slanted}
	\newtheorem{theo}{Theorem}[section]
	\newtheorem{prop}[theo]{Proposition}

	\newtheorem{lemma}[theo]{Lemma}
	\newtheorem{definition}[theo]{Definition}

	\def\ind#1{\mathbbmss{1}_{#1}}
	
	\def\Id{\mathop{\mbox{Id}}}

 	\newcommand{\tend}[2]{\xrightarrow[#1\to#2]{}}
 	\newcommand{\converge}[3]{\xrightarrow[#1\to#2]{#3}}
	\newcommand{\esp}[2][{}]{{\EE}_{#1}\!\!\left[#2\right]} 
	\newcommand{\espc}[3][{}]{{\EE}_{#1}\left[#2\,\left|\,#3\right.\right]}
	\def\ind#1{\mathbbmss{1}_{#1}}

\renewcommand{\X}{\mathbf{X}}
\newcommand{\Xt}{\widetilde{\X}}
\newcommand{\Xc}{\X_{\C}}
\newcommand{\Cb}{\overline{\C}}
\newcommand{\Xb}{\overline{\X}}

\begin{document}
\bibliographystyle{amsplain}

\title{Notes on Austin's multiple ergodic theorem}
\author{Thierry de la Rue}
\date{}

\maketitle

\begin{abstract}
	The purpose of this note is to present my understanding of Tim Austin's proof of the multiple ergodic theorem for commuting transformations, emphasizing on the use of joinings, extensions and factors. The existence of a sated extension, which is a key argument in the proof, is presented in a general context.
\end{abstract}

\section{Introduction}
The norm convergence of multiple ergodic averages for commuting transformations (Theorem~\ref{theo:main} below) was first proved in 2008 by Terence Tao~\cite{Tao2008}. We intend to present here the quite different proof proposed by Tim Austin~\cite{Austin2008} using the machinery of joinings, extensions and factors. This text is written after Austin's talk at the conference \textit{Dynamical Systems and Randomness}, (held in Paris, Institut Henri Poincaré, May 2009), and a short conversation with him following his talk.
The major part is, as far as I understand it, quite faithful to Austin's original proof. The only slightly original contribution is the proof of the existence of a sated extension, which is presented in a general context.

\begin{theo}
\label{theo:main}
 Let $d\ge 1$, and $T_1,\ldots,T_d$ be $d$ commuting, measure-preserving invertible transformations of the standard Borel probability space $(X,\A,\mu)$. Then for any choice of $f_1,\ldots,f_d\in L^\infty(\mu)$, the multiple ergodic averages
$$ \dfrac{1}{N}\sum_{n=1}^{N}f_1\circ T_1^n\cdots f_d\circ T_d^n $$
converge in $L^2(\mu)$ as $N\to\infty$.
\end{theo}
 
\subsection*{The strategy}
The case $d=1$ corresponds to the standard ergodic theorem of Von Neumann, and in this case the limit is clearly identified, as the orthogonal projection of the function $f_1$ on the subspace of $L^2$-functions which are measurable with respect to the factor $\sigma$-algebra 
$$\I^{T_1} := \Bigl\{A\in\A:\ \mu(A\Delta T_1^{-1}A)=0\Bigr\}.$$ 
(This factor $\sigma$-algebra is called the \textit{isotropy factor} by Austin; isotropy factors play a crucial role here and we will use the above notation for several transformations in the sequel.)

The proof for $d\ge2$ is presented by induction on $d$: We assume that $d\ge2$ is such that Theorem~\ref{theo:main} has already been proved up to the case of $d-1$ commuting transformations. Then we identify a simple class $\C$ of systems $(X,\A,\mu,T_1,\ldots,T_d)$ with $d$ commuting transformations for which the desired result is easily deduced from the $(d-1)$-case. Next step consists in the introduction of a larger class of systems, the so-called \textit{$\C$-sated systems}. (Note that the notion of \emph{satedness} is not explicit in~\cite{Austin2008}: It has been formalized in a subsequent work by Austin~\cite{Austin2009} dealing also with some polynomial sequences.) The $\C$-sated systems are characterized by a quite simple structure of their joinings with any $\C$-system, and this enables us to prove for them the theorem, using the induction hypothesis and  a version of Van der Corput lemma. Finally, and this is the point where the machinery of joinings plays its crucial role, we show that any system possesses an extension which is $\C$-sated. Since we know that the theorem holds for $\C$-sated systems, it obviously holds for all their factors, hence for all systems.

For the sake of simplicity, we first present the induction step passing from one to two commuting transformations, admitting the existence of a $\C$-sated extension for any system. Then we will see how the same argument can be generalized to pass from $d-1$ to $d$ transformations. Finally, in a completely independent section, we prove a general result on joinings showing why any system admits a $\C$-sated extension, achieving the proof of Theorem~\ref{theo:main}.

\section{The case of two commuting transformations}
\label{sec:d=2}
In all this section, we assume $d=2$ and we present the argument showing how the theorem for two commuting transformations can be proved, using the well-known result in the case of a single transformation.

\subsection{$\C$-systems}
We first observe that there are two very simple cases in which the convergence
in $L^2$ of the ergodic averages
\begin{equation}
\label{eq:average,d=2}
 \dfrac{1}{N}\sum_{n=1}^{N}f_1\circ T_1^n f_2\circ T_2^n
\end{equation}
is a trivial consequence of the single-transformation case:
\begin{itemize}
 \item If $T_1=\Id$, which amounts to saying that the isotropy factor $\I^{T_1}$ is the whole $\sigma$-algebra $\A$: The above ergodic average reduces to $f_1$ times an ergodic average for the single transformation $T_2$. Obviously, it is enough that $f_1$ be $\I^{T_1}$-measurable to get this reduction.
\item If $T_1=T_2$, in other words if the isotropy factor $\I^{T_2T_1^{-1}}$ is the whole $\sigma$-algebra $\A$: Then \eqref{eq:average,d=2} reduces to an ergodic average for the product $f_1f_2$ and the single transformation $T_1=T_2$. Note that this reduction holds as soon as $f_1$ is measurable with respect to $\I^{T_2T_1^{-1}}$.
\end{itemize}

Now, we introduce the class of $\C$-systems as the class of systems $\X=(X,\A,\mu,T_1,T_2)$ for which
$$ \A=\I^{T_1}\vee\I^{T_2T_1^{-1}}. $$
In other words, $\X$ is a $\C$-system if it is isomorphic to a joining of two systems of the form $\mathbf{X_1}=(X_1,\A_1,\mu_1,\Id,S)$ and $\mathbf{X_2}=(X_2,\A_2,\mu_2,T,T)$. In a system of the class $\C$, any bounded measurable function $f_1$ can be arbitrarily well approximated in $L^2$ by a finite sum of products of the form $gh$, where $g$ is $\I^{T_1}$-measurable and $h$ is $\I^{T_2T_1^{-1}}$-measurable. For each such term $g h$, we can simultaneously apply the two reductions explained above, and we get that the $L^2$-convergence of the ergodic averages~\eqref{eq:average,d=2} holds in any $\C$-system.

\subsection{$\C$-sated systems}

\subsubsection{Looking for characteristic factors}
In any system $\X=(X,\A,\mu,T_1,T_2)$, it is now natural to consider the factor $\sigma$-algebra 
$$ \Xc := \I^{T_1}\vee\I^{T_2T_1^{-1}}. $$
Considering the action of $T_1,T_2$ on this factor, we obviously get a $\C$-system which we also denote by $\Xc$. It is straightforward to check that, in fact, this factor is the largest factor in $\X$ (in the sense of inclusion of $\sigma$-algebras) on which the action of the transformation gives rise to a $\C$-system: We call it the \textit{largest $\C$-factor of $\X$}.

Observe that if $f_1$ is measurable with respect to $\Xc$, the same argument as above proves the convergence in $L^2$ of the ergodic averages \eqref{eq:average,d=2}. Now, we are looking for simple conditions on $\X$ ensuring that, when studying the convergence of these ergodic averages, we can replace $f_1$ by its projection $\EE[f_1|\Xc]$, which would immediately lead to the desired conclusion. In other words, we are looking for conditions implying
\begin{equation}
\label{eq:convergence to 0}
  \left\| \dfrac{1}{N}\sum_{n=1}^{N}f_1\circ T_1^n f_2\circ T_2^n \right\|_{L^2} \tend{N}{\infty} 0 \quad \mbox{as soon as }\EE[f_1|\Xc]=0.
\end{equation}
(Although we shall not explicitely use here the notion of \textit{characteristic factors}, we can note that the above condition is equivalent to ``$(\Xc,\A)$ is a pair of characteristic factors for the convergence of the ergodic averages \eqref{eq:average,d=2}'', see Definition 4.1 in \cite{Austin2008}.)

\subsubsection{Van der Corput lemma}

To obtain the convergence to 0 in \eqref{eq:convergence to 0}, we will make use of the following lemma.

\begin{lemma}[Van der Corput]
\label{lemma:Van der Corput}
 Let $(u_n)$ be a bounded sequence in a Hilbert space. If
$$ \lim_{H\to\infty}\lim_{N\to\infty} \dfrac{1}{H} \sum_{h=1}^{H} \dfrac{1}{N} \sum_{n=1}^{N} <u_n u_{n+h}> = 0, $$
then
$$ \lim_{N\to\infty}  \left\| \dfrac{1}{N} \sum_{n=1}^{N} u_n \right\| = 0. $$
\end{lemma}

For the sake of completeness, a proof of this lemma is included in Annex~A.

\subsubsection{A sufficient condition for the convergence}

In view of Lemma~\ref{lemma:Van der Corput}, we are led to study the expression
$$ \dfrac{1}{H} \sum_{h=1}^{H} \dfrac{1}{N} \sum_{n=1}^{N} \int_X f_1\circ T_1^n\, f_2\circ T_2^n \, \overline{f_1}\circ T_1^{n+h} \, \overline{f_2}\circ T_2^{n+h} d\mu. $$
Using the invariance of $\mu$ with respect to $T_1$, we can rewrite each integral in the form
$$ \int_X f_1 \, \overline{f_1}\circ T_1^h \, (f_2 \, \overline{f_2} \circ T_2^h)\circ (T_2 T_1^{-1})^n d\mu. $$
For each fixed $h$, the usual ergodic theorem for the single transformation $T_2 T_1^{-1}$ gives
\begin{multline*}
 \lim_{N\to\infty}\dfrac{1}{N}\sum_{n=1}^N \int_X f_1  \, \overline{f_1}\circ T_1^h  \, (f_2  \, \overline{f_2} \circ T_2^h)\circ (T_2 T_1^{-1})^n d\mu \\
= \int_X f_1  \, \overline{f_1}\circ T_1^h \, \esp{f_2  \, \overline{f_2} \circ T_2^h\, |\, \I^{T_2 T_1^{-1}}} d\mu. 
\end{multline*}
Now, it is convenient to  view the latter integral as 
\begin{equation}
 \label{eq:self-joining}
\int_{X\times X} f_1(x_1) \, \overline{f_1 (T_1^h x_1)}  \, f_2(x_2) \, \overline{f_2 (T_2^h x_2)} \, d (\mu\otimes_{\I^{T_2 T_1^{-1}}}\mu) (x_1,x_2),
\end{equation}
where $\mu\otimes_{\I^{T_2 T_1^{-1}}}\mu$ denotes the relatively independent self-joining of $\X$ over its factor $\sigma$-algebra $\I^{T_2 T_1^{-1}}$. Remark that this probability distribution is invariant under the action of the transformation $\widetilde{T_1}:=T_1\otimes T_2:\ (x_1,x_2)\mapsto (T_1x_1, T_2x_2)$. Indeed, if $\phi_1$ and $\phi_2$ are bounded measurable functions on $X$, we have
\begin{align*}
 &\int_{X\times X} \phi_1(T_1x_1)\, \phi_2(T_2x_2) d (\mu\otimes_{\I^{T_2 T_1^{-1}}}\mu) (x_1,x_2)\\
	=  &\int_X \esp{\phi_1\circ T_1|\I^{T_2 T_1^{-1}}} \esp{\phi_2\circ T_2|\I^{T_2 T_1^{-1}}} d\mu \\
	=  &\int_X \esp{\phi_1|\I^{T_2 T_1^{-1}}}\circ T_1\, \esp{\phi_2|\I^{T_2 T_1^{-1}}}\circ T_2\, d\mu(x) \quad \parbox[t]{5cm}{(because $\I^{T_2 T_1^{-1}}$ is invariant by both transformations)}\\
	=  &\int_X \esp{\phi_1|\I^{T_2 T_1^{-1}}}\circ T_1\, \esp{\phi_2|\I^{T_2 T_1^{-1}}}\circ T_1\, d\mu(x) \quad \parbox[t]{5cm}{(since on $\I^{T_2 T_1^{-1}}$, $T_1$ and $T_2$ coincide)}\\
	=  &\int_X \esp{\phi_1|\I^{T_2 T_1^{-1}}}\, \esp{\phi_2|\I^{T_2 T_1^{-1}}}\, d\mu(x) \\
	=  &\int_{X\times X} \phi_1(x_1)\, \phi_2(x_2) d (\mu\otimes_{\I^{T_2 T_1^{-1}}}\mu) (x_1,x_2)\\
\end{align*}
Of course, $\mu\otimes_{\I^{T_2 T_1^{-1}}}\mu$ is also invariant by the transformation $\widetilde{T_2}:=T_2\otimes T_2$. This gives us a big system with two commuting transformations $\Xt:=\left(X\times X, \A\otimes\A, \mu\otimes_{\I^{T_2 T_1^{-1}}}\mu, \widetilde{T_1}, \widetilde{T_2}\right)$. Observe that our original system $\X$ is a factor of $\Xt$, which is obtained by considering only the first coordinate.

We have to average \eqref{eq:self-joining} over $h$. Writing this average in $\Xt$ gives
$$ \int_{X\times X} f_1\otimes f_2 (x_1,x_2)\, \dfrac{1}{H}\sum_{h=1}^H \overline{f_1\otimes f_2} \left(\widetilde{T_1}^h(x_1,x_2)\right) d (\mu\otimes_{\I^{T_2 T_1^{-1}}}\mu) (x_1,x_2). $$
Applying the usual ergodic theorem for $\widetilde{T_1}=T_1\otimes T_2$ in $\Xt$, we see that the latter expression converges, as $H\to\infty$, to
$$ \int_{X\times X} f_1\otimes f_2\ \espc[\Xt]{\overline{f_1\otimes f_2}}{\I^{\widetilde{T_1}}} d (\mu\otimes_{\I^{T_2 T_1^{-1}}}\mu). $$
An obvious sufficient condition for this limit to vanish is $\espc[\Xt]{f_1\otimes f_2}{\I^{T_1\otimes T_2}}=0$. The following proposition links this condition with the largest $\C$-factor $\Xt_\C$ of $\Xt$.

\begin{prop}
 If $\espc[\Xt]{f_1(x_1)}{\Xt_\C}=0$, then $\espc[\Xt]{f_1\otimes f_2}{\I^{T_1\otimes T_2}}=0$.
\end{prop}

\begin{proof}
Observe that $\I^{\widetilde{T_1}}$ is a $\C$-factor of $\Xt$, hence we have $\I^{\widetilde{T_1}}\subset \Xt_\C$ by definition of the maximal $\C$-factor. Therefore, a sufficient condition for the above conclusion to hold is $\espc[\Xt]{f_1\otimes f_2}{\Xt_\C}=0$. Moreover, we can also note that $(x_1,x_2)\mapsto f_2(x_2)$ is also $\Xt_\C$-measurable: Indeed, on the $\sigma$-algebra generated by the second coordinate, the transformations $\widetilde{T_1}$ and $\widetilde{T_2}$ coincide, hence this $\sigma$-algebra is itself a $\C$-factor of $\Xt$. We can then write
$$ \espc[\Xt]{f_1\otimes f_2}{\Xt_\C} = f_2(x_2)\, \espc[\Xt]{f_1(x_1)}{\Xt_\C}, $$
from which the proposition follows immediately.
\end{proof}

\subsubsection{$\C$-sated systems}
From the above reasoning, we see that if our system $\X$ is such that
\begin{equation}
 \label{eq:pleasant}
 \espc[\X]{f_1}{\X_\C}=0 \Longrightarrow \espc[\Xt]{f_1(x_1)}{\Xt_\C},
\end{equation}
then the convergence in $L^2$ of the ergodic averages~\eqref{eq:average,d=2} holds. Observe that in the RHS of~\eqref{eq:pleasant}, we are considering together two important factors of the big system $\Xt$: On the one hand the factor generated by the first coordinate (which, as already mentioned, is nothing but the original system $\X$), and on the other hand the largest $\C$-factor $\Xt_\C$ of $\X$. In other words, we are considering a joining of our system $\X$ with the $\C$-system $\Xt_\C$. This motivates the following fundamental definition.

\begin{definition}
\label{def:sated}
 The system $\X$ is said to be \emph{$\C$-sated} if, for any joining $\lambda$ of $\X$ with a $\C$-system $\mathbf{Y}$ and any bounded measurable function $f$ on $\X$, we have
\begin{equation}
 \label{eq:C-sated}
\espc[\lambda]{f(x)}{\mathbf{Y}} = \EE_{\lambda}\Bigl[ \, \espc[\X]{f(x)}{\Xc} \,\Big|\, \mathbf{Y} \Bigr].
\end{equation}
In other words, the system is $\C$-sated if any joining of $\X$ with a $\C$-system is relatively independent over the largest $\C$-factor $\Xc$ of $\X$.
\end{definition}
Of course, any $\C$-sated system satisfies~\eqref{eq:pleasant}. Hence, a partial conclusion up to this point can be stated as follows:

\begin{prop}
 If $\X=(X,\A,\mu,T_1,T_2)$ is a $\C$-sated system, then the ergodic averages~\eqref{eq:average,d=2} converge in $L^2$ for any choice of $f_1$ and $f_2$ in $L^\infty(\mu)$.
\end{prop}

%

\section{The case of $d$ commuting transformations}
\label{sec:general d}
In this section, we assume that $d\ge2$ is such that Theorem~\ref{theo:main} has already been proved in the case of $d-1$ commuting transformations, and we adapt the arguments of the preceding section to see how to prove the ergodic theorem in the case of $\C$-sated systems of $d$ commuting transformations.

In this general case, we define the class of $\C$-systems as the class of systems $\X=(X,\A,\mu,T_1,\ldots,T_d)$ for which
$$ \A=\I^{T_1}\vee\I^{T_2T_1^{-1}}\vee\cdots\vee\I^{T_dT_1^{-1}}. $$

A factor $\sigma$-algebra of a system $\X=(X,\A,\mu,T_1,\ldots,T_d)$ on which the action of $T_1,\ldots,T_d$ defines a $\C$-system will be called a \emph{$\C$-factor of $\X$}. 
In any system $\X$ there always exists a largest $\C$-factor
$$ \Xc := \I^{T_1}\vee\I^{T_2T_1^{-1}}\vee\cdots\vee\I^{T_dT_1^{-1}}. $$
($\Xc$ is defined as a factor $\sigma$-algebra of $\X$, but we will use the same notation to denote the $\C$-system obtained by considering the action of $T_1,\ldots,T_d$ on this sub-$\sigma$-algebra.)

\smallskip

The same argument as in the case $d=2$ proves that convergence in $L^2$ of ergodic averages 
\begin{equation}
\label{eq:average}
 \dfrac{1}{N}\sum_{n=1}^{N}f_1\circ T_1^n \cdots f_d\circ T_d^n
\end{equation}
reduces in the class of $\C$-systems to the case of $d-1$ commuting transformations, and in fact it is enough for this reduction to be valid that $f_1$ be measurable with respect to $\Xc$. Hence we are just looking for conditions ensuring that we can replace $f_1$ in~\eqref{eq:average} by its projection $\espc[\X]{f_1}{\Xc}$.

\smallskip

The class of $\C$-sated systems is defined word for word as in Definition~\ref{def:sated}. Assuming now that $\X=(X,\A,\mu,T_1,\ldots,T_d)$ is a $\C$-sated system, we have to show that for any choice of $f_1,\ldots,f_d\in L^\infty(\mu)$,
\begin{multline}
\label{eq:to prove}
 \espc[\X]{f_1}{\Xc}=0\ \Longrightarrow \\
\lim_{H\to\infty} \lim_{N\to\infty}
\dfrac{1}{H} \sum_{h=1}^{H} \dfrac{1}{N} \sum_{n=1}^{N} 
\int_X f_1\circ T_1^n\cdots f_d\circ T_d^n \, \overline{f_1}\circ T_1^{n+h} \cdots \overline{f_d}\circ T_d^{n+h} d\mu =0,
\end{multline}
which in turn, by Lemma~\ref{lemma:Van der Corput}, implies
$$
\left\| \dfrac{1}{N}\sum_{n=1}^{N}f_1\circ T_1^n \cdots f_d\circ T_d^n \right\|_{L^2} \tend{N}{\infty} 0.
$$


\subsection{Furstenberg self-joining}

Since we have assumed the validity of Theorem~\ref{theo:main} for $d-1$ commuting transformations, the averages
$$ \dfrac{1}{N}\sum_{n=1}^{N} \int_X g_1\circ T_1^n\, g_2\circ T_2^n \cdots g_d\circ T_d^n \, d\mu =
\dfrac{1}{N}\sum_{n=1}^{N} \int_X g_1 \, g_2 \circ (T_2 T_1^{-1})^n \cdots g_d\circ (T_d T_1^{-1})^n \, d\mu
$$
converge for any choice of $g_1,\ldots,g_d$ in $L^\infty(\mu)$. Applying this convergence in the case of indicator functions $g_i=\ind{A_i}$, $i=1,\ldots,d$, it is standard to see that the limit defines a probability measure $\lambda$ on $X^d$ by the formula 
$$ \lambda(A_1\times\cdots\times A_d) := \lim_{N\to\infty} \dfrac{1}{N}\sum_{n=1}^{N} \int_X \ind{A_1}(T_1^n x) \cdots\ind{A_d}(T_d^n x) \, d\mu(x). $$
Moreover, it is straightforward to check that $\lambda$ enjoys the following properties:
\begin{itemize}
 \item $\lambda$ is invariant by the transformation $T_i\otimes\cdots\otimes T_i$ for all $i=1,\ldots,d$;
 \item The $d$ marginal distributions of $\lambda$ are equal to $\mu$;
 \item $\lambda$ is also invariant by the transformation $T_1\otimes T_2\otimes\cdots\otimes T_d$.
\end{itemize}
The above first two properties together mean that $\lambda$ is a $d$-fold self-joining of the system $\X=(X,\A,\mu,T_1,\ldots,T_d)$. this self-joining was introduced by Furstenberg in~\cite{Furstenberg77}, and therefore refered to as \textit{Furstenberg self-joining} by Austin.

As in the case of two commuting transformations, we now define a big system 
$$\Xt:=(X^d,\A^{\otimes d},\lambda,\widetilde{T_1},\widetilde{T_2},\ldots,\widetilde{T_d}),$$
where $\widetilde{T_1}:=T_1\otimes T_2\otimes\cdots\otimes T_d$ and, for $2\le i\le d$, $\widetilde{T_i}:=T_i\otimes T_i\otimes\cdots\otimes T_i$. Observe by considering the first coordinate that $\X$ is a factor of $\Xt$. Note also that, for $2\le i\le d$, $\widetilde{T_1}$ and $\widetilde{T_i}$ coincide on the sub-$\sigma$-algebra $\A_i$ generated by the $i$-th coordinate, from which we can deduce
\begin{equation}
\label{eq:in the C factor}
 \A_2\otimes\cdots\otimes\A_d \subset \Xt_{\C}.
\end{equation}

We now turn back to~\eqref{eq:to prove}. For any fixed $h$, we have by definition of Furstenberg self-joining
\begin{multline}
\dfrac{1}{N} \sum_{n=1}^{N} 
\int_X f_1\circ T_1^n\cdots f_d\circ T_d^n \, \overline{f_1}\circ T_1^{n+h} \cdots \overline{f_d}\circ T_d^{n+h} d\mu\\
  \tend{N}{\infty}  
\int_{X^d} (f_1\otimes \cdots\otimes f_d)\, (\overline{f_1\otimes \cdots\otimes f_d}\circ\widetilde{T_1}^h)\, d\lambda. \end{multline}
Averaging the latter expression over $h\in\{1,\ldots,H\}$, and letting $H$ go to infinity gives, using the mean ergodic theorem for the single transformation $\widetilde{T_1}$,
\begin{equation}
 \label{eq:must vanish}
\int_{X^d} (f_1\otimes \cdots\otimes f_d)\ \espc[\Xt]{\overline{f_1\otimes \cdots\otimes f_d}}{\I^{\widetilde{T_1}}} d\lambda.
\end{equation}
Assume now that $\espc[\X]{f_1}{\Xc}=0$. Then if $\X$ is a $\C$-sated system, we have
$$ \espc[\Xt]{f_1(x_1)}{\Xt_{\C}} = \espc[\X]{f(x)}{\X_{\C}} = 0. $$
Recalling \eqref{eq:in the C factor}, we get
$$ \espc[\Xt]{f_1\otimes\cdots\otimes f_d}{\Xt_{\C}} = f_2(x_2)\cdots f_d(x_d)\ \espc[\Xt]{f_1(x_1)}{\Xt_{\C}} = 0, $$ 
and since $\I^{\widetilde{T_1}}\subset \Xt_{\C}$, 
$$ \espc[\Xt]{f_1\otimes \cdots\otimes f_d}{\I^{\widetilde{T_1}}} = 0, $$
which proves~\eqref{eq:to prove}.

\smallskip

\goodbreak

We thus have proved the following partial result:
\begin{prop}
 If the statement of Theorem~\ref{theo:main} is valid for $d-1$ commuting transformations, then it is also valid for any $\C$-sated system of $d$ commuting transformations.
\end{prop}

\section{Existence of $\C$-sated extensions}
\label{sec:sated extensions}

The purpose of this section is to prove the existence of a $\C$-sated extension for any dynamical system, in a general context including the case we need to finish the proof of Theorem~\ref{theo:main}. An important part of the arguments used below was developped in~\cite{L-R-dlR2003} for the study of another class of systems, namely the class of all factors of all countable self-joinings of a given system. But as mentionned in~\cite{dlR2006}, they work in a quite general setting which we present here in details.

\smallskip

From now on, let $\C$ denote a class of dynamical systems, which we always assume to be stable under taking isomorphisms. As before, we call \emph{$\C$-factor} of a dynamical system $(X,\A,\mu,T_1,\ldots,T_d)$ any factor sub-$\sigma$-algebra on which the action of $T_1,\ldots,T_d$ defines a system in the class $\C$. In the particular case of class $\C$ used in the preceding sections, it was quite obvious to see that any system admits a largest $\C$-factor. this is in fact a general result provided a stability assumption on $\C$.

\begin{lemma}
 If the class $\C$ is stable under taking countable joinings, then any system $\X$ admits a largest $\C$-factor, which we denote by $\Xc$.
\end{lemma}

\begin{proof}
We just set
$$ \Xc := \{A\in \A\ :\ \mbox{$A$ belongs to some $\C$-factor of }\X\}\;, $$
and we claim that it is a $\C$-factor. Since $(X,\A,\mu)$ is a standard Borel space, the $\sigma$-algebra $\A$ equipped with the metric
$d(A,B):= \mu(A\Delta B)$ is separable (where we naturally identify subsets $A$ and $A'$ of $X$ when $\mu(A\Delta A')=0$). Therefore there exists a countable family $(A_i)_{i\in I}$ dense in $\Xc$, and for each $i$ there is some $\C$-factor $\F_i$ containing $A_i$. Since the class $\C$ is stable under taking countable joinings, $\F:=\bigvee_{i\in I}\F_i$ is itself a $\C$-factor. By density, we have $\Xc \subset\F$ but, since $\Xc$ contains every $\C$-factors, we have $\Xc=\F$.
\end{proof}

If $\C$ is stable under taking countable joinings, we can thus repeat Definition~\ref{def:sated} in this more general setting:
\begin{definition}
\label{def:sated-bis}
 The system $\X$ is said to be \emph{$\C$-sated} if any joining of $\X$ with a $\C$-system is relatively independent over the largest $\C$-factor $\Xc$ of $\X$.
\end{definition}

\begin{prop}
\label{prop:always sated}
 If the class $\C$ is stable under taking countable joinings and under taking factors, then any system $\X$ is $\C$-sated.
\end{prop}
 The proof is based on a fundamental lemma, published simultaneously in two papers~\cite{GlasnerThouvenotWeiss2000,LemanczykParreauThouvenot2000}, stating that if two systems $\X$ and $\mathbf{Y}$ are not disjoint, then $\X$ possesses a non-trivial common factor with a joining of countably many copies of $\mathbf{Y}$. We slightly rephrase this lemma in order to make it more convenient for our purposes:

\begin{lemma}
\label{lemma:fundamental}
 Let $\lambda$ be a joining of two systems $\X=(X,\A,\mu,(T_j))$ and $\mathbf{Y}=(Y,\B, \nu,(S_j))$, and let $g$ be a bounded measurable function defined in $\mathbf{Y}$. Then there exists a factor sub-$\sigma$-algebra $\F$ in $\X$ such that the action of $T_1,\ldots,T_d$ on $\F$ is isomorphic to a factor of some joining of countably many copies of $\mathbf{Y}$, and satisfying 
\begin{equation}
 \label{eq:fundamental lemma}
\espc[\lambda]{g(y)}{\X} = \espc[\lambda]{g(y)}{\F\otimes\{\emptyset,Y\}}.
\end{equation}
\end{lemma}

\begin{proof}
We consider a countable family of copies of the dynamical system defined by the joining $\lambda$, and consider their relatively independent joining $\lambda_\infty$ over their common factor $\X$. Then $\lambda_\infty$ is a probability measure on the space $X\times Y^{\NN}$, which is easily seen to be invariant under the shift transformation on each $x$-fiber, $\sigma:(x,y_0,y_1,y_2,\ldots)\mapsto(x,y_1,y_2,y_3,\ldots)$. Moreover, $\lambda_\infty$ conditioned on each such fiber is a product measure. A relative version of Kolmogorov 0-1 law (see \textit{e.g.} \cite{LemanczykParreauThouvenot2000}, Lemma 9) gives that, modulo $\lambda_\infty$, the $\sigma$-algebra $\I^{\sigma}$ of shift-invariant events coincides with  the $\sigma$-algebra $\A\otimes\left\{\emptyset,Y^{\NN}\right\}$ generated by the $x$ coordinate. Consider now a bounded measurable function $g$ on $Y$, and set $g_\infty(x,(y_n)):=g(y_1)$ for $x\in X$ and $(y_n)\in Y^{\NN}$. Applying the ergodic theorem in the dynamical system $(X\times Y^{\NN},\lambda_\infty,\sigma)$ to the function $g_\infty$, we obtain 
$$ \dfrac{1}{N}\sum_{n=1}^{N} g(y_n) \converge{N}{\infty}{L^2(\lambda_\infty)}
\espc[\lambda_\infty]{g_\infty}{\I^\sigma} = \espc[\lambda_\infty]{g(y_1)}{\X},$$
and by definition of $\lambda_\infty$ the latter is equal to $\espc[\lambda]{g(y)}{\X}$. 
Hence, $\espc[\lambda]{g(y)}{\X}$ coincides modulo $\lambda_\infty$ with a function which is measurable with respect to $(y_n)_{n\in\NN}$. It follows that the factor $\F$ of $\X$ generated by $\espc[\lambda]{g(y)}{\X}$ is isomorphic to a factor of the joining of countably many copies of $\mathbf{Y}$ obtained by considering the $(y_n)$-coordinates in $\lambda_\infty$. Finally, with this definition of $\F$, we obviously have~\eqref{eq:fundamental lemma}. 
\end{proof}

\begin{proof}[Proof of Proposition~\ref{prop:always sated}]
Let $\X$ be any dynamical system, and $\lambda$ be a joining of $\X$ with a $\C$-system $\mathbf{Y}$. For a given bounded measurable function $g$ defined in $\mathbf{Y}$, let $\F$ be the factor sub-$\sigma$-algebra given by Lemma~\ref{lemma:fundamental}. By stability of $\C$ under taking countable joinings and factors, $\F$ is a $\C$-factor of $\X$, and $\F$ is therefore contained in the largest $\C$-factor $\Xc$. Equation~\eqref{eq:fundamental lemma} then gives
$$
\espc[\lambda]{g(y)}{\X} = \espc[\lambda]{g(y)}{\Xc\otimes\{\emptyset,Y\}},
$$
%
and this equation means that in the joining $\lambda$, $\X$ and $\mathbf{Y}$ are relatively independent over $\Xc$. This proves that $\X$ is $\C$-sated.
\end{proof}

The class $\C$ of dynamical systems which is used in Section~\ref{sec:d=2}, and its generalization in Section~\ref{sec:general d}, are easily proved to be stable under taking countable joinings, but unfortunately they are not stable under taking factors (see Annex~B). 

This is why it is necessary in general to pass to extensions to get $\C$-sated systems. The remaining of the section is devoted to the proof of the following theorem.

\begin{theo}
\label{theo:sated extension}
 Let $\C$ be a class of dynamical systems which is stable under taking countable joinings. Then any system admits a $\C$-sated extension.
\end{theo}

For $\C$ satisfying the hypothesis of the above theorem, we start by introducing the class $\Cb$ consisting of dynamical systems which are factors of $\C$-systems. Obviously $\Cb$ is stable by taking factors, and we can also check that $\Cb$ is stable under taking countable joinings. Indeed, let $\overline{\mathbf{Z}}$ be a joining of a countable family $(\Xb_i)_{i\in I}$ of $\Cb$-systems. For each $i$, let $\X_i$ be a $\C$-extension of $\Xb_i$, and define $\mathbf{Y}_i$ as the relatively independent joining of $\X_i$ and $\overline{\mathbf{Z}}$ over their common factor $\Xb_i$. Then, consider the relatively independent joining $\mathbf{Z}$ of the $\mathbf{Y}_i$'s over their common factor $\overline{\mathbf{Z}}$. In $\mathbf{Z}$, each factor $\Xb_i$ of $\overline{\mathbf{Z}}$ is identified with a factor of $\X_i$, hence $\overline{\mathbf{Z}}$ itself, which is generated by all the $\Xb_i$'s, is contained in the $\sigma$-algebra generated by the $\X_i$'s. $\overline{\mathbf{Z}}$ is thus a factor of the joining of the $\X_i$'s defined by $\mathbf{Z}$, and since $\C$ is stable under taking countable joinings, this joining is a $\C$-system.

In any system $\X=(X,\A,\mu,(T_j))$, there exist therefore a largest $\C$-factor $\Xc$, and a largest $\Cb$-factor $\X_{\Cb}$. Since any $\C$-system is obviously a $\Cb$-system, $\Xc\subset\X_{\Cb}$. 

\begin{prop}
\label{prop:characterization}
 $\X$ is $\C$-sated if and only if $\Xc=\X_{\Cb}$.
\end{prop}
\begin{proof}
 The \emph{if} part is a direct corollary of Proposition~\ref{prop:always sated} applied to the class $\Cb$. Conversely, let us assume that $\X$ is $\C$-sated. Let $\mathbf{Y}$ be a $\C$-extension of $\X_{\Cb}$, and consider the relatively independent joining of $\X$ and $\mathbf{Y}$ over their common factor $\X_{\Cb}$: Since $\X$ is $\C$-sated, this joining is relatively independent over $\Xc$. But this is only possible if $\X_{\Cb}\subset\Xc$. 
\end{proof}

\begin{proof}[Proof of Theorem~\ref{theo:sated extension}]
We use the same construction as above: Given a dynamical system $\X$, we consider its largest $\Cb$-factor $\X_{\Cb}$, a $\C$-extension $\mathbf{Y}$ of $\X_{\Cb}$, and the relatively independent joining of $\X$ and $\mathbf{Y}$ over their common factor $\X_{\Cb}$. Let us denote by $\mathbf{Z}$ the latter system: $\mathbf{Z}$ is the extension of $\X$ which will be proved to be $\C$-sated. For this, by Proposition~\ref{prop:characterization} it is enough to establish that the largest $\Cb$-factor of $\mathbf{Z}$ is $\mathbf{Y}$: Since $\mathbf{Y}$ is a $\C$-system, this will give $\mathbf{Z}_\C=\mathbf{Z}_{\Cb}$.

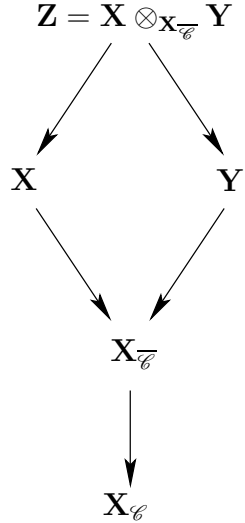
\begin{center}
\begin{picture}(0,0)%
\includegraphics{diagram.pstex}%
\end{picture}%
\setlength{\unitlength}{4144sp}%
\begingroup\makeatletter\ifx\SetFigFont\undefined%
\gdef\SetFigFont#1#2#3#4#5{%
  \reset@font\fontsize{#1}{#2pt}%
  \fontfamily{#3}\fontseries{#4}\fontshape{#5}%
  \selectfont}%
\fi\endgroup%
\begin{picture}(1305,3162)(-1067,1097)
\put(-899,4088){\makebox(0,0)[lb]{\smash{{\SetFigFont{12}{14.4}{\rmdefault}{\mddefault}{\updefault}{\color[rgb]{0,0,0}$\mathbf{Z}=\mathbf{X}\otimes_{\X_{\Cb}}\mathbf{Y}$}%
}}}}
\put(-499,1162){\makebox(0,0)[lb]{\smash{{\SetFigFont{12}{14.4}{\rmdefault}{\mddefault}{\updefault}{\color[rgb]{0,0,0}$\Xc$}%
}}}}
\put(-1052,3107){\makebox(0,0)[lb]{\smash{{\SetFigFont{12}{14.4}{\rmdefault}{\mddefault}{\updefault}{\color[rgb]{0,0,0}$\mathbf{X}$}%
}}}}
\put(181,3101){\makebox(0,0)[lb]{\smash{{\SetFigFont{12}{14.4}{\rmdefault}{\mddefault}{\updefault}{\color[rgb]{0,0,0}$\mathbf{Y}$}%
}}}}
\put(-453,2089){\makebox(0,0)[lb]{\smash{{\SetFigFont{12}{14.4}{\rmdefault}{\mddefault}{\updefault}{\color[rgb]{0,0,0}$\mathbf{X}_{\Cb}$}%
}}}}
\end{picture}%

\end{center}

Let us consider a joining $\lambda$ of $\mathbf{Z}$ with a $\C$-system $\mathbf{W}$. Since the joining of $\mathbf{Y}$
and $\mathbf{W}$ induced by $\lambda$ is still a $\C$-system, Proposition~\ref{prop:always sated} ensures that, inside $\lambda$, $\X$ and $(\mathbf{W}\vee\mathbf{Y})$ are relatively independent over $\X_{\Cb}$. Hence $\X$ and $\mathbf{W}$ are relatively independent over $\mathbf{Y}$, and finally $\mathbf{Z}$ and $\mathbf{W}$ are relatively independent over $\mathbf{Y}$ (because $\mathbf{Z}$ is generated by $\X$ and $\mathbf{Y}$). We thus have proved that any joining of $\mathbf{Z}$ with a $\C$-system is relatively independent over $\mathbf{Y}$. Taking in particular the relatively independent joining of $\mathbf{Z}$ with a $\C$-extension of $\mathbf{Z}_{\Cb}$, we see that this is only possible if $\mathbf{Z}_{\Cb}\subset\mathbf{Y}$. Since the converse inclusion obviously holds, this concludes the proof.
\end{proof}

\subsection*{Annex A. Proof of Van der Corput Lemma}

Here is a proof of Lemma~\ref{lemma:Van der Corput}. First, observe that, since the sequence $(u_n)$ is bounded, for any $H$ we have
$$ \left\| \dfrac{1}{N}\sum_{n=1}^N u_n \right\| =  \left\| \dfrac{1}{N}\sum_{n=1}^N \dfrac{1}{H}\sum_{h=1}^H u_{n+h} \right\| + O(H/N). $$
Using the classical inequality $(a+b)^2\le 2(a^2+b^2)$, then the triangular inequality and finally the Cauchy-Schwartz inequality in the form $(1/N \sum_1^N a_n)^2\le 1/N\sum_1^Na_n^2$, we get
\begin{align*}
 \left\| \dfrac{1}{N}\sum_{n=1}^N u_n \right\|^2
	& \le 2 \left\| \dfrac{1}{N}\sum_{n=1}^N \dfrac{1}{H}\sum_{h=1}^H u_{n+h} \right\|^2 + O(H^2/N^2) \\
	& \le 2 \left(\dfrac{1}{N}\sum_{n=1}^N \left\| \dfrac{1}{H}\sum_{h=1}^H u_{n+h} \right\|\right)^2 + O(H^2/N^2) \\
	& \le  \dfrac{2}{N}\sum_{n=1}^N \left\| \dfrac{1}{H}\sum_{h=1}^H u_{n+h} \right\|^2 + O(H^2/N^2) \\
\end{align*}
We now have to estimate
$$ \dfrac{1}{N}\sum_{n=1}^N \left\| \dfrac{1}{H}\sum_{h=1}^H u_{n+h} \right\|^2 = 
	\dfrac{1}{N}\sum_{n=1}^N \dfrac{1}{H}\sum_{h=1}^H \dfrac{1}{H}\sum_{h'=1}^H <u_{n+h} , u_{n+h'}>. $$
We split the RHS into three pieces $P_{h=h'}$, $P_{h<h'}$, and $P_{h>h'}$, corresponding respectively to the terms where $h=h'$, $h<h'$, and $h>h'$. 
The first piece is simply controlled by choosing $H$ large enough:
$$ P_{h=h'} = \dfrac{1}{N}\sum_{n=1}^N \dfrac{1}{H^2}\sum_{h=1}^H <u_{n+h} , u_{n+h}> = O(1/H). $$
The second and third pieces are treated with the same computation, we only detail here the case $(h<h')$:
\begin{align*} 
P_{h<h'} &= \dfrac{1}{H}\sum_{h=1}^H \dfrac{1}{H}\sum_{h'=h+1}^H \dfrac{1}{N}\sum_{n=1}^N <u_{n+h} , u_{n+h'}> \\
	 &= \dfrac{1}{H}\sum_{h=1}^H \dfrac{1}{H}\sum_{h'=1}^{H-h} \dfrac{1}{N}\sum_{n=1}^N <u_{n} , u_{n+h'}> + O(H/N)\\
	 &= \dfrac{1}{H}\sum_{h=1}^{H-1} \dfrac{1}{H}\sum_{h'=1}^h \dfrac{1}{N}\sum_{n=1}^N <u_{n} , u_{n+h'}> + O(H/N)\\
\end{align*}
Fixing $H$ large enough, the hypothesis then implies that $P_{h<h'}$ can be made arbitrarily close to zero when $N\to\infty$, which achieves the proof.

\subsection*{Annex B. A factor of a $\C$-system is not always a $\C$-system}

Here is an example showing that the class $\C$ defined in Section~\ref{sec:d=2} is not stable under taking factors. For each $\alpha\in\TT:=\RR/\ZZ$, let us denote by $R_\alpha$ the translation on $\TT$: $x\mapsto x+\alpha\mod 1$, and by $\mu$ the Haar measure on $\TT$. For some fixed irrational $\alpha$, we consider the system $\X=(\TT\times\TT,\mu\otimes\mu,T_1,T_2)$ where $T_1:=R_\alpha\otimes R_{2\alpha}$, and $T_2:=R_{2\alpha}\otimes R_{2\alpha}$. Denoting by $x$ (respectively $y$) the first (respectively second) coordinate on $\TT\times\TT$, we observe that any function of $2x-y\mod1$ is invariant by $T_1$, hence is measurable with respect to the $\C$-factor $\Xc$. Observe also that on the $\sigma$-algebra generated by $y$, $T_1$ and $T_2$ define the same action, hence any function of $y$ is also measurable with respect to the $\C$-factor $\Xc$. It follows that the factor of $\X$ generated by $2x\mod1=(2x-y)+y\mod1$ is contained in $\Xc$, hence is a factor of a $\C$-system. However, the action of $(T_1,T_2)$ restricted to this factor is isomorphic to the action of $(R_{2\alpha},R_{4\alpha})$ on $\TT$. The latter is certainly not a $\C$-system, since both $\I^{R_{2\alpha}}$ and $\I^{R_{4\alpha}R_{2\alpha}^{-1}}$ are trivial.

\bibliography{met.bib}

\end{document}